\documentclass[12pt,a4paper]{amsart}
\usepackage[english]{babel}
\usepackage{lmodern}
\usepackage{newlfont}
\usepackage{booktabs}
\usepackage{color}
\usepackage[T1]{fontenc}
\usepackage[utf8]{inputenc}
\usepackage{indentfirst}
\usepackage{amsmath}
\usepackage{amsmath,amssymb}
\usepackage{amsthm}
\usepackage{geometry}
\usepackage{mathtools}
\usepackage{listings}
\usepackage{amsfonts}
\usepackage{braket}
\usepackage{emptypage}
\usepackage{newlfont}
\usepackage{amssymb}
\usepackage{graphicx}
\usepackage[italian]{varioref}
\usepackage{accents}
\usepackage{comment}
\usepackage{stmaryrd}
\usepackage{pgfplots}

\makeatletter
\labelformat{equation}{\tagform@{#1}}
\makeatother
\usepackage{hyperref}
\usepackage{relsize}
\usepackage{faktor}
\usepackage{enumerate}
\usepackage{fancyhdr}

\usepackage[colorinlistoftodos]{todonotes}

\renewcommand{\phi}{\varphi}
\renewcommand{\epsilon}{\varepsilon}

\newcommand{\numberset}{\mathbb}

\newcommand{\e}{\varepsilon}

\newcommand{\R}{\numberset{R}}

\newcommand{\Om}{\Omega}

\theoremstyle{definition}
\newtheorem{definition}{Definition}[section]
\theoremstyle{definition}
\newtheorem{definizione}{Definizione}[section]
\theoremstyle{definition}
\newtheorem{rmk}[definizione]{Remark}
\theoremstyle{plain}
\newtheorem{thm}[definizione]{Theorem}
\theoremstyle{plain}
\newtheorem{prop}[definizione]{Proposition}
\theoremstyle{plain}
\newtheorem{lemma}[definizione]{Lemma}
\theoremstyle{plain}
\newtheorem{cor}[definizione]{Corollary}
\theoremstyle{definition}

\numberwithin{equation}{section}

\begin{document}
\title[An extension of Cabr\'{e}-Chanillo..]{An extension of Cabr\'{e}-Chanillo theorem to the $p$-laplacian}
\thanks{Massimo Grossi was supported by INDAM-GNAMPA. Zexi Wang was supported by the China Scholarship Council (202406990064), and Chongqing Graduate Student Research Innovation Project (CYB25100).}

\author[M. Grossi]{Massimo Grossi}
\address[Massimo Grossi] {Dipartimento di Scienze di Base e Applicate per 1'Ingegneria, Universit\`{a} di Roma ``La Sapienza", Via Scarpa 10, 00161 Roma, Italy.}
\email{massimo.grossi@uniroma1.it}

\author[L. Montoro]{Luigi Montoro}
\address[Luigi Montoro] {Dipartimento di Matematica e Informatica, Universit\`{a} della Calabria
 Ponte Pietro Bucci 31B, 87036 Arcavacata di Rende, Cosenza, Italy.}
\email{luigi.montoro@unical.it}

\author[B. Sciunzi]{Berardino Sciunzi}
\address[Berardino Sciunzi] {Dipartimento di Matematica e Informatica, Universit\`{a} della Calabria
 Ponte Pietro Bucci 31B, 87036 Arcavacata di Rende, Cosenza, Italy.}
\email{berardino.sciunzi@unical.it }



\author[Z. Wang]{Zexi Wang}
\address[Zexi Wang] {School of Mathematics and Statistics, Southwest University,
Chongqing, 400715, People's Republic of China.}\email{zxwangmath@163.com}

\maketitle

\begin{abstract}
{ In this paper, we study the critical points of stable solutions for the following $p$-laplacian equation
\begin{equation*}
\begin{cases}
-div\big(|\nabla u|^{p-2}\nabla u\big)=f(u)&in\ \Om,\\
u>0&in\ \Om,\\
u=0&on\ \partial\Om,
\end{cases}
\end{equation*}
where $p>2$, $f\in C^1([0,+\infty))$ satisfies $f(t)>0$ for $t>0$,  and $\Om\subset\R^2$ is a smooth bounded domain with non-negative curvature of the boundary.
Via a suitable approximation argument, we prove that, a stable solution $u$ admits, as its only critical point, the internal absolute maxima and possibly saddle points with zero index. Moreover, $Argmax(u)$ is a point or segment.
 }

\vspace{.2cm}

\vspace{.2cm}

\end{abstract}

\section{Introduction and statement of the main result}
In this paper, we study the following quasi-linear problem
\begin{equation}\label{1}
\begin{cases}
-div\big(|\nabla u|^{p-2}\nabla u\big)=f(u)&in\ \Om,\\
u>0&in\ \Om,\\
u=0&on\ \partial\Om,
\end{cases}
\end{equation}
where $p>1$, $\Om\subset\R^2$ is a smooth bounded domain and $f$ is a suitable reaction.

For $p=2$, in the case of $f \equiv1$, problem \ref{1} is called the {torsion problem}. Makar-Limanov \cite{M} proved that if $\Om$ is convex, then the solution $u$ of \ref{1} has only one critical point and the level sets of $u$ are strictly convex. The convexity assumption is difficult to relax, indeed in \cite{GG}, the authors gave some examples of domains ``close to'' (in a suitable sense) a convex one with a large number of critical points. Moreover, this result is sharp in terms of the nonlinearity, since Hamel et al. \cite{HNS} constructed a solution which is not quasi-concave with a more general nonlinearity $f$ in a ``like-stadium'' domain. Here we call a function is {quasi-concave} if its super-level sets are all
convex.
The result of \cite{M} has been extended to any dimension by Korevaar and Lewis \cite{KL}.
More recently, Gallo and Squassina \cite{GS} generalized the result of \cite{M} to the case of 
a sequence of nonlinearities $f_n(x)$ satisfying $f_n(x)\rightarrow f_\infty$ in $\Om$ as $n\rightarrow\infty$, where $f_\infty$ is a positive constant.

Concerning the {eigenvalue problem}, which means $f(u)=\lambda u$, $\lambda$ is the first eigenvalue of the Laplacian with zero Dirichlet boundary condition.
Here it was proved by Brascamp and Lieb \cite{BL} and Acker et al. \cite{APP} in dimension $N=2$ that if $\Om\subset {\R}^N$ is strictly convex, then the
first eigenfunction $u$ is log-concave, that is $\log u$ is concave.
Moreover, Caffarelli and Friedman \cite{CF2} obtained the uniqueness and non-degeneracy of the critical point in dimension two.
For more investigations in this case, we can see \cite{DG} for the number of critical points of the second eigenfunction
in convex planar domains.

Now we consider the case of a general nonlinearity $f$. Gidas et al. \cite{GNN} proved the uniqueness and non-degeneracy of the critical point under the assumption that $\Om\subset \mathbb{R}^N$ ($N\geq 3$) is symmetric with respect to a point and just convex in any direction.  Since then, some conjectures have claimed that the uniqueness of the critical point holds in more general convex domains without the symmetry assumption. A good class of solutions to extend the result of \cite{GNN} is that of the semi-stable solutions. We recall that $u$ is a semi-stable solution of problem \ref{1} if the first eigenvalue of the linearized operator $-\Delta- f'(u)$ in $\Om$ is non-negative, or equivalently if
\begin{equation*}
\int_{\Omega}|\nabla \phi|^2-f'(u)\phi^2\geq 0,
\end{equation*}
for any $\phi \in C_0^\infty(\Omega)$.
 An
important result concerning this class of solutions is given by Cabr\'{e} and Chanillo \cite{CC}, where they proved the uniqueness and non-degeneracy of the critical point of semi-stable solutions in convex planar domains with  boundary of positive curvature. This result was extended to the case of vanishing curvature in \cite{DGM} by degree theory. 
We also mention  the recent papers  \cite{DG24} and \cite{GP}  where the authors considered the uniqueness and non-degeneracy of the critical point for semi-stable solutions on convex domains of Robin boundary and Riemannian surfaces, respectively.

However, if $p\neq 2$, 
the $p$-laplace operator is degenerate ($p>2$) or singular ($1<p<2$) in the critical set
\begin{equation*}
  \mathcal{Z}=\{x\in\Om:\nabla u(x)=0\}.
\end{equation*}
The solutions of \ref{1} are generally of class $C^{1,\alpha}$ with $\alpha\in (0,1)$, not $C^{2}$, and solve \ref{1} only in the weak sense, see \cite{Di,Li,T}.
This is
the best regularity that one can expect for solutions to nonlinear equations involving
the $p$-Laplacian.
Thus the above techniques cannot be directly applied to problem \ref{1} when $p\neq 2$. A classical idea, thus, is to regularize the operator, apply the result and pass to the limit. 
This is what has been done in \cite{S} to prove the concavity properties of solutions to \ref{1} with $f(u)=1$ and $f(u)=u^{p-1}$. For more classical results regarding the regularized procedure, the readers may refer to \cite{ACF,CDS,CFV1,GS} and the references therein. 

 When $\Omega$ is a ball, by performing the moving plane, Damascelli and Sciunzi \cite{DS} proved the uniqueness and non-degeneracy of the critical point for the weak solutions.


As far as we know, there are only a few works dealing with the critical points of solutions to \ref{1} with a general reaction term in general convex domains when $p\neq2$, besides the one \cite{BMS2},
where the authors proved
 the uniqueness of the critical point of quasi-concave solutions to \ref{1} in convex bounded domains of the plane.
 So, inspired by the work  \cite{CC} and \cite{DGM}, in this paper, we are concerned with the critical points of the {\em stable solutions} of \ref{1}.
 Moreover, we are particularly interested in the case $p>2$.

Before stating our main result, we recall the definition of stable solutions for $p$-laplace equations.
  \begin{definition}\label{defi}\cite[Definition 1.1]{CS'} 
  Let $u\in C^{1}(\bar{\Om})$ be a weak solution  of \ref{1} with $p>2$, we say that $u$ is {stable}, 
  if the first eigenvalue of the linearized operator $L(u)$ in $\Om$ is positive, or equivalently if
   \begin{equation*}
   \langle L(u)\varphi,\phi\rangle >0, 
\end{equation*}
for any $\varphi\in W_0^{1,p}(\Om)\backslash \{0\}$, 
where $L(u)$ is defined by, for any  $\varphi,\psi\in W_0^{1,p}(\Om)$,
\begin{equation*}
\langle L(u)\varphi,\psi\rangle=\int_{\Om}|\nabla u|^{p-2}(\nabla \phi\nabla\psi)+(p-2)|\nabla u|^{p-4}(\nabla u\nabla \phi)(\nabla u\nabla\psi)-f'(u)\phi\psi.
\end{equation*}
  \end{definition}
Denote the associated energy functional of \ref{1}  by $J:W_0^{1,p}(\Om)\to\R$,
\begin{equation*}
  J(v)=\int_\Omega\frac{|\nabla v|^p}{p}- F(v),
\end{equation*}
where $F(s)=\int_0^sf(t)dt$. 
Then $J$ is well defined in $W_0^{1,p}(\Om)$, and of class $C^2$. Moreover,
it's easy to find that $\langle L(u)\varphi,\phi\rangle$ is the second variation of $J$ at $u$.




Our first result states as follows.
\begin{thm}\label{th1}
Assume   
$f\in C^1([0,+\infty))$ satisfies $f(t)>0$ for $t>0$,
and 
$\Om\subset\R^2$ is a smooth bounded domain whose  boundary has {positive curvature}.
Suppose that $u$ 
is a stable solution of problem \ref{1}.
Then the critical points of $u$ has only the internal absolute maxima and possibly saddle points with zero index. Moreover, $Argmax(u)$ is a point or {segment}.
\end{thm}


\begin{rmk}\label{rmk1}
{\rm In this paper, applying the implicit function theorem to a suitable operator $L(\varepsilon,v)$ in the space $E$, see \ref{space}-\ref{3}, for any fixed $\varepsilon\in (-\varepsilon_0,\varepsilon_0)$,
we find a function $u_\varepsilon\in C^1(\bar{\Om})$ such that $L(\varepsilon,u_\varepsilon)=0$.  
 Moreover, $u_\varepsilon\rightarrow u$ in $C^1(\bar{\Om})$ as $\varepsilon\rightarrow0$. By the defintion of $L(\varepsilon,v)$ and using some standard regularity results, we have that $u_\varepsilon\in C^{2,\alpha}(\bar{\Om})$ is a positive stable solution for
some Dirichlet boundary problem \ref{2}.
Then, considering the critical points of $u_\varepsilon$ and passing to the limit, we
complete the proof.
}
\end{rmk}


\begin{rmk}
{\rm
As mentioned in Remark \ref{rmk1}, since we want to use the implicit function theorem to construct an
approximation solution $u_\varepsilon$, the first eigenvalue of $L(u)$ in $\Om$ must be non-zero. 
Moreover, to obtain
a stable solution $u_\varepsilon$, we cannot start from $u$ with the first eigenvalue negative, otherwise, the first eigenvalue of the linearized operator $L_\varepsilon(u_\varepsilon)$ defined in \ref{diff} for the approximating problem could become negative.}
\end{rmk}

Moreover, we have the following result allowing the curvature of $\partial \Om$ to vanish somewhere.
\begin{thm}\label{th2'}
Assume 
$f\in C^1([0,+\infty))$ satisfies $f(t)>0$ for $t>0$,
and 
$\Om\subset\R^2$ is a smooth bounded domain whose  boundary has {non-negative curvature} such that the subset of zero-curvature consists of isolated points or segments.
Suppose that $u$ 
is a stable solution of problem \ref{1}.
Then the critical points of $u$ has only the internal absolute maxima and possibly saddle points with zero index. Moreover, $Argmax(u)$ is a point or {segment}.
\end{thm}

The paper is organized as follows. 
In Section \ref{sec2}, we 
approximate $u$ with $u_\varepsilon$. Section \ref{sec3} is devoted to the proof of Theorem \ref{th1}. Finally, in Section \ref{sec4}, we prove Theorem \ref{th2'}.

\section{The approximation argument}\label{sec2}
In this section, we will use the implicit function theorem to construct an approximation solution to problem \ref{1}.
Since $u$ is a weak solution of \ref{1}, by definition, we have $J'(u)=0$, where $J'$ is the differential of $J$, that is,
\begin{equation*}
  \langle J'(u),\varphi\rangle=\int_\Om|\nabla u|^{p-2}(\nabla u\nabla \phi)- f(u)\phi,\quad \text{for any $\varphi\in W_0^{1,p}(\Om)$}.
\end{equation*}
 For any $\varepsilon>0$ small enough,  
and $v$ in a $W_0^{1,p}$-neighborhood of $u$ (if $v=u$, set $\varepsilon=0$), i.e., there exist  $\varepsilon_0,\delta>0$ such that $\varepsilon\in (-\varepsilon_0,\varepsilon_0)$ and $\|v-u\|_{W_0^{1,p}(\Om)}< \delta$.   
Define
\begin{align}\label{space}
  E:=\big\{(\varepsilon,w)\in \mathbb{R}\times W_0^{1,p}(\Om):\varepsilon\in (-\varepsilon_0,\varepsilon_0),\|w-u\|_{W_0^{1,p}(\Om)}< \delta\big\},
\end{align}
we consider the operator $L(\varepsilon,v):E\rightarrow W^{-1,p'}(\Om)$ defined by
\begin{equation}\label{3}
  L(\varepsilon,v):=J'_\varepsilon(v),
\end{equation}
where $p'$ is the conjugate exponent of $p$, $J'_\varepsilon$ is the differential of the $C^2$-functional $J_\e:W_0^{1,p}(\Om)\to\R$,
\begin{equation*}
  J_\varepsilon(v)=
  \int_\Omega\frac{\left(\e^2+|\nabla v|^2\right)^\frac p2}{p}- F(v),
\end{equation*}
namely,
\begin{equation*}
  \langle J'_\varepsilon(v),\varphi\rangle=\int_\Om\big(\e^2 +|\nabla v|^2\big)^\frac{p-2}2(\nabla v\nabla \phi)-f(v)\phi,\quad \text{for any $\varphi\in W_0^{1,p}(\Om)$}.
\end{equation*}




For \ref{space}-\ref{3}, we have the following existence result which is very important in this paper.

\begin{prop}\label{toget}
Under the assumptions of Theorem \ref{th1} or \ref{th2'}, for any fixed $\varepsilon\in (-\varepsilon_0,\varepsilon_0)$, problem
 \begin{equation}\label{2}
\begin{cases}
-div\Big[\big(\e^2 +|\nabla v|^2\big)^\frac{p-2}2\nabla v\Big]=f(v)&in\ \Om,\\
v=0&on\ \partial\Om,
\end{cases}
\end{equation}
has a unique positive solution $u_\varepsilon\in C^{2,\alpha}(\bar{\Om})$. Moreover,  $u_\varepsilon\rightarrow u$ in $C^1(\bar{\Om})$ as $\varepsilon\rightarrow0$, 
and $u_\varepsilon$ is stable. 
\end{prop}


\begin{proof}
 In order to apply the implicit function theorem, we need to prove

 \begin{description}
\item [(i)] $L(0,u)=0$.\\
\item [(ii)] $L(\varepsilon,v)$ is continuous in $E$.\\
 \item [(iii)] The derivative ${\partial_ v}L(\varepsilon,v)$ exists and is continuous in $E$. \\
\item [(iv)] ${\partial_ v}L(0,u)$ is invertible.
\end{description}

\noindent  If (i)-(iv) hold,
for any fixed $\varepsilon\in (-\varepsilon_0,\varepsilon_0)$, by the implicit function theorem, we get that, there exists a unique $u_\varepsilon\in W_0^{1,p}(\Om)$ such that
$L(\varepsilon,u_\varepsilon)=0$, i.e., $J_\varepsilon'(u_\varepsilon)=0$.
That is, problem \ref{2} admits a unique weak solution $u_\varepsilon$. Moreover, $u_\varepsilon\rightarrow u$ in $W_0^{1,p}(\Om)$ as $\varepsilon\rightarrow 0$.

 Let us verify (i)-(iv).
 Firstly, (i) holds because by the assumptions of $u$, we have $(0,u)\in E$ and
\begin{equation*}
  L(\varepsilon,v)|_{(\varepsilon,v)=(0,u)}=J'_\varepsilon(v)|_{(\varepsilon,v)=(0,u)}=J'(u)=0.
\end{equation*}
 Next, we consider (ii). 
By using the Lebesgue dominated convergence theorem, it is straightforward to prove that
\begin{align*}
  \int_{\Om}\big(\e_n^2 +|\nabla v_n|^2\big)^\frac{p-2}2(\nabla v_n\nabla \phi)
  \rightarrow \int_{\Om}\big(\e^2 +|\nabla v|^2\big)^\frac{p-2}2(\nabla v\nabla \phi),
\end{align*}
and
\begin{equation*}
 \int_{\Om} f(v_n)\phi\rightarrow \int_{\Om} f(v)\phi,
\end{equation*}
for any $\phi\in W_0^{1,p}(\Om)$, as $\varepsilon_n\rightarrow\varepsilon$ and $v_n\rightarrow v$ in $W_0^{1,p}(\Om)$. Thus
\begin{equation*}
  \langle J'_{\varepsilon_n}(v_n),\varphi\rangle\rightarrow \langle J'_\varepsilon(v),\varphi\rangle,\quad \text{for any $\phi\in W_0^{1,p}(\Om)$},
\end{equation*}
which means $J'_{\varepsilon_n}(v_n)\rightarrow J'_{\varepsilon}(v)$
as $\varepsilon_n\rightarrow\varepsilon$ and $v_n\rightarrow v$ in $W_0^{1,p}(\Om)$.
This proves (ii). 
Similarly, 
we can prove that (iii) holds.
Finally, by the stability of $u$, namely,
\begin{equation*}
\int_{\Om}|\nabla u|^{p-2}|\nabla \phi|^2+(p-2)|\nabla u|^{p-4}(\nabla u\nabla \phi)^2
-f'(u)\phi^2>0,
\end{equation*}
we deduce that ${\partial_ v}L(0,u)$ is invertible,
which gives (iv).   

Since $u_\varepsilon\rightarrow u$ in $W_0^{1,p}(\Om)$ as $\varepsilon\rightarrow 0$, we have $\|u_\varepsilon\|_{W_0^{1,p}(\Om)}\leq {C}_1$ with $C_1>0$ independent of $\varepsilon$.
By $p>2$, using the Sobolev embedding theorem, we have
\begin{equation*}
  \|u_\varepsilon\|_{L^\infty(\Omega)}\leq C_2\|u_\varepsilon\|_{W_0^{1,p}(\Om)}\leq C_1C_2,
\end{equation*}
for some positive constant ${C}_2$ independent of $\varepsilon$.
Hence, by the uniform $L^\infty$ estimate, we
applying \cite[Proposition 3.1, Lemma 4.1]{CFV1} and
\cite[Theorem 1]{Li} to deduce that $u_\varepsilon\in C^{1,\beta}(\bar{\Om})$ for some $\beta\in (0,1)$. In addition, $\beta$ is independent of $\varepsilon$ and $\|u_\varepsilon\|_{C^{1,\beta}(\bar{\Om})}$ is uniformly bounded in $\varepsilon$. Consequently, by Arzel\`{a}-Ascoli Theorem, $u_\varepsilon$ converges to a function $u_*$ in $C^1(\bar{\Om})$  as $\varepsilon\rightarrow 0$. Then $u_\varepsilon\rightarrow u_*$ in $W_0^{1,p}(\Om)$ as $\varepsilon\rightarrow 0$. By the uniqueness of the limit, we have $u=u_*$.

Notice that $f(u_\varepsilon)\in C^{1,\beta}({\bar{\Om}})$, the standard regularity results (see \cite[Theorem 6.6]{GT}) give at least that $u_\varepsilon\in C^{2,\beta}(\bar{\Om})$.
 Moreover, by the Hopf boundary lemma (see \cite[Lemma A.3]{S}), we have $\frac{\partial u}{\partial \nu}<0$, thus $\nabla u\neq\mathbf{0}$ at any point on $\partial \Om$, where $\nu$ denotes the unit exterior normal vector to $\partial \Om$. This together with $u>0$ in $\Om$  and $u_\varepsilon\rightarrow u$ in $C^1(\bar{\Om})$ yields that $u_\varepsilon>0$ in $\Om$.
Obviously, the stability
of $u$ implies 
\begin{equation*}
  \langle L_\varepsilon(u_\varepsilon)\varphi,\varphi\rangle>0,
\end{equation*}
for any $\phi \in W_0^{1,p} (\Om)\backslash \{0\}$, where $L_\varepsilon(u_\varepsilon)$ is
 the linearized operator of \ref{2} at $u_\varepsilon$, defined by, for any $\phi,\psi\in W_0^{1,p}(\Om)$,
 \begin{align}\label{diff}
\langle L_\varepsilon(u_\varepsilon)\varphi,\psi\rangle=&\int_{\Om}\big(\e^2 +|\nabla u_\varepsilon|^2\big)^\frac{p-2}2(\nabla \phi\nabla\psi)+
(p-2)\big(\e^2 +|\nabla u_\varepsilon|^2\big)^\frac{p-4}2
\nonumber\\
&\times (\nabla u_\varepsilon\nabla \phi)(\nabla u_\varepsilon\nabla\psi)-{{f'(u_\varepsilon)}\phi\psi}.
\end{align}
So $u_\varepsilon$ is stable, and we complete the proof.
\end{proof}

\section{Proof of Theorem \ref{th1}}\label{sec3}


For any fixed $\varepsilon\in (-\varepsilon_0,\varepsilon_0)$, let $u_\varepsilon$ be given in Proposition \ref{toget}, we have
\begin{thm}\label{th2}
Under the assumptions of Theorem \ref{th1}, $u_\varepsilon$ has a unique critical point, which is a non-degenerate maximum point in the sense that the Hessian of $u$ at this point is negative definite.
\end{thm}

To prove Theorem \ref{th2}, as in \cite[Section 2]{CC}, we introduce the following notation: for any $\theta\in [0,2\pi)$, we write $\mathbf{{e_\theta}}=(\cos\theta, \sin \theta)$, and set
 \begin{equation*}
 N_\theta=\big\{x\in \bar \Om: u_{\varepsilon,\theta}(x)=\langle \nabla u_\varepsilon(x),\mathbf{{e_\theta}}\rangle=0\big\},
 \end{equation*}
\begin{equation*}
  M_\theta=\big\{x\in N_\theta: \nabla u_{\varepsilon,\theta}(x)=D^2 u_\varepsilon(x)\cdot \mathbf{{e_\theta}}=\mathbf{0}\big\}.
\end{equation*}
Moreover, by a generalized Hopf boundary lemma (see \cite[Theorem 1.1]{Sa}), we have $\frac{\partial u_\varepsilon}{\partial \nu}<0$, thus $\nabla u_\varepsilon\neq\mathbf{0}$ on $\partial \Om$. 
Recall that $u_\e\in C^{3}(\bar \Om)$ by the standard regularity theory, 
then the curvature at $x\in \partial \Om$ is given by
\begin{equation*}
  \mathfrak{R}(x)=-\frac{u_{\varepsilon,x_2x_2}u_{\varepsilon,x_1}^2-2u_{\varepsilon,x_1x_2}u_{\varepsilon,x_1}u_{\varepsilon,x_2}+u_{\varepsilon,x_1x_1}u_{\varepsilon,x_2}^2}{|\nabla u_\varepsilon|^3}.
\end{equation*}
The following result tells us that the nodal sets $N_\theta$ are $C^2$ curves in $\bar \Om$ without self-intersection, and any critical point of $u_\varepsilon$ is non-degenerate.

\begin{prop}\label{prop1}
Under the assumptions of Theorem \ref{th1}, for
any $\theta\in [0,2\pi)$, the nodal set $N_\theta$ is a $C^2$ curve in $\bar \Om$ without self-intersection, which hits $\partial \Om$ at the two end points of $N_\theta$. Moreover, in any critical point of $u_\varepsilon$, the Hessian has rank $2$.
\end{prop}

\begin{proof}
We adopt the idea of \cite{CC} to complete our proof. First, using the implicit function theorem, we immediately have

\vspace{.1cm}

$\bullet$ {\em Property 1: around any point $x\in (N_\theta\cap \Om)\backslash M_\theta$, the nodal set $N_\theta$ is a $C^2$ curve.}

\vspace{.1cm}

Since the positivity of the curvature on the boundary implies the
strictly convexity of $\Om$,
arguing as the proof of \emph{Property 2} in \cite[Pages 4-5]{CC}, using the implicit function theorem again, we obtain

\vspace{.1cm}

$\bullet$ {\em Property 2: $M_\theta \cap \partial \Om=\emptyset$, $N_\theta \cap \partial \Om$ consists of exactly two points $p_1, p_2$, and around each $p_i$, $N_\theta$ is a $C^2$ curve that intersects $\partial \Om$ transversally at $p_i$, $i=1,2$.}

\vspace{.1cm}

Differentiating \ref{2} with respect to the direction $\mathbf{{e_\theta}}$, we obtain
\begin{align*}
  -div\Big[\big(\e^2 +|\nabla u_\varepsilon|^2\big)^\frac{p-2}2\nabla u_{\varepsilon,\theta}+
 (p-2) \big(\e^2 +|\nabla u_\varepsilon|^2\big)^\frac{p-4}2(\nabla u_\varepsilon &\nabla u_{\varepsilon,\theta})\nabla u_\varepsilon
  \Big]\\
  =&f'(u_\varepsilon)u_{\varepsilon,\theta} \quad in \ \Om.
\end{align*}
 Notice that
 \begin{align*}
  & \big(\e^2 +|\nabla u_\varepsilon|^2\big)^\frac{p-2}2\nabla u_{\varepsilon,\theta}+
 (p-2) \big(\e^2 +|\nabla u_\varepsilon|^2\big)^\frac{p-4}2(\nabla u_\varepsilon \nabla u_{\varepsilon,\theta})\nabla u_\varepsilon
\\=&\big(\e^2 +|\nabla u_\varepsilon|^2\big)^\frac{p-4}2A(x)\nabla u_{\varepsilon,\theta},
 \end{align*}
 where
 \begin{equation*}
 A(x)=(a_{ij}(x))=\left(
  {\begin{array}{cccc}
   \displaystyle\e^2 +|\nabla u_\varepsilon|^2+(p-2)u_{\varepsilon,x_1}^2& (p-2)u_{\varepsilon,x_1}u_{\varepsilon,x_2}\\
  (p-2)u_{\varepsilon,x_1}u_{\varepsilon,x_2}&\e^2 +|\nabla u_\varepsilon|^2+(p-2)u_{\varepsilon,x_2}^2
 \end{array}}
 \right).
\end{equation*}
For any $\xi=(\xi_1,\xi_2)\in \mathbb{R}^2$, 
we have
\begin{align*}
&a_{11}(x)\xi_1^2+a_{22}(x)\xi_2^2+2a_{12}(x)\xi_1\xi_2 \nonumber\\
=
  &\big(\e^2 +|\nabla u_\varepsilon|^2+(p-2)u_{\varepsilon,x_1}^2\big) \xi_1^2+\big(\e^2 +|\nabla u_\varepsilon|^2+(p-2)u_{\varepsilon,x_2}^2\big) \xi_2^2 \nonumber\\&+2(p-2)u_{\varepsilon,x_1}u_{\varepsilon,x_2}\xi_1\xi_2 \\
=&  \big(\e^2 +|\nabla u_\varepsilon|^2\big) |\xi|^2+(p-2)(u_{\varepsilon,x_1}\xi_1+u_{\varepsilon,x_2}\xi_2)^2
 \nonumber \\
 \leq & \big(\e^2 +(p-1)|\nabla u_\varepsilon|^2\big) |\xi|^2 \leq (p-1)\big(\e^2 +|\nabla u_\varepsilon|^2\big) |\xi|^2\nonumber,
\end{align*}
and
\begin{align*}
a_{11}(x)\xi_1^2+a_{22}(x)\xi_2^2+2a_{12}(x)\xi_1\xi_2 \geq \big(\e^2 +|\nabla u_\varepsilon|^2\big) |\xi|^2.
\end{align*}
Thus,
\begin{align}\label{eq12}
\varepsilon^{{p-2}}|\xi|^2\leq \big(\e^2 +|\nabla u_\varepsilon|^2\big)^\frac{p-2}2|\xi|^2 &\leq \big(\e^2 +|\nabla u_\varepsilon|^2\big)^\frac{p-4}2\sum\limits_{i,j=1}^2a_{ij}(x)\xi_i\xi_j \nonumber \\
&\leq (p-1)\big(\e^2 +|\nabla u_\varepsilon|^2\big)^\frac{p-2}2|\xi|^2\leq C|\xi|^2.
\end{align}
In addition, for any $x,y\in \Om$,  there exists $\Lambda>0$ such that,
\begin{equation}\label{eq3}
  |a_{ij}(x)-a_{ij}(y)|\leq \Lambda |x-y|,\quad i,j=1,2.
\end{equation}
For instance, using the mean value theorem, we have
\begin{align*}
  |a_{11}(x)-a_{11}(y)|=&|\nabla u_\varepsilon(x)|^2-|\nabla u_\varepsilon(y)|^2+(p-2)\big[u_{\varepsilon,x_1}^2(x)-u_{\varepsilon,x_1}^2(y)\big]\\
  \leq & C_1\big[|\nabla u_\varepsilon(x)|-|\nabla u_\varepsilon(y)|\big]+C_1\big[u_{\varepsilon,x_1}(x)-u_{\varepsilon,x_1}(y)\big]\\
  \leq & C_2 u_{\varepsilon,x_hx_k}(z)|_{z=\vartheta x+(1-\vartheta)y}|x-y|\leq
 \Lambda |x-y|, \quad h,k=1,2,
\end{align*}
for some $\vartheta\in [0,1]$. Therefore, for any $i,j=1,2$,
\begin{equation*}
  \big(\e^2 +|\nabla u_\varepsilon|^2\big)^\frac{p-4}2|a_{ij}(x)-a_{ij}(y)|\leq \big(C+\varepsilon^{p-4}\big)|x-y|.
\end{equation*}
Since $|f'(u_\varepsilon)|\leq {C}$, we have that, at any point $x\in N_\theta \cap \Omega$, $u_{\varepsilon,\theta}$ vanishes but not of infinite order. Otherwise, by  \ref{eq12} and \ref{eq3}, using the strong unique continuation theorem (see \cite[Theorem]{A} or \cite[Theorem 1.1]{GL}), 
we obtain $u_{\varepsilon,\theta}\equiv 0$ in $\Om$, which contradicts $N_\theta \cap \partial \Om=\{p_1,p_2\}$.
Moreover, we have
\begin{align}\label{f1}
  &-\big(\e^2 +|\nabla u_\varepsilon|^2\big)^\frac{p-2}2\Delta u_{\varepsilon,\theta}-(p-2) \big(\e^2 +|\nabla u_\varepsilon|^2\big)^\frac{p-4}2(\nabla u_\varepsilon \nabla u_{\varepsilon,\theta})\Delta u_\varepsilon \nonumber\\
  &-(p-2)(p-4) \big(\e^2 +|\nabla u_\varepsilon|^2\big)^\frac{p-6}2(\nabla u_\varepsilon \nabla u_{\varepsilon,\theta})\sum_{i,j=1}^2
 u_{\varepsilon,x_i}u_{\varepsilon,x_j}u_{\varepsilon,x_ix_j}
 \nonumber \\
  &-2(p-2) \big(\e^2 +|\nabla u_\varepsilon|^2\big)^\frac{p-4}2\sum_{i,j=1}^2
  u_{\varepsilon,\theta,x_i}u_{\varepsilon,x_j}u_{\varepsilon,x_ix_j}
  \\&-(p-2) \big(\e^2 +|\nabla u_\varepsilon|^2\big)^\frac{p-4}2\sum_{i,j=1}^2
 u_{\varepsilon,x_i}u_{\varepsilon,x_j}u_{\varepsilon,\theta,x_ix_j}
  =f'(u_\varepsilon)u_{\varepsilon,\theta}.\nonumber
\end{align}
Using {\cite[Theorem I]{B}}, we know that, around any point $x_0\in N_\theta\cap \Omega$, $u_{\varepsilon,\theta}$ behaves locally as a homogeneous polynomial $p_m(x)$ of degree $m\geq 1$, which satisfies
\begin{align*}
 a_{11}(x_0)p_{m,x_1x_1}(x)+a_{22}(x_0)p_{m,x_2x_2}(x)+2a_{12}(x_0)p_{m,x_1x_2}(x)=0.
\end{align*} 
By a direct computation, it holds
\begin{align*}
  |A(x_0)|&=a_{11}(x_0)a_{22}(x_0)-a_{12}^2(x_0)
  \\&=\varepsilon^4+p\varepsilon^2 |\nabla u_\varepsilon(x_0)|^2+(p-1)|\nabla u_\varepsilon(x_0)|^4>0.
\end{align*}
So, by a change of coordinates, 
$p_m(x)$ is harmonic. Moreover, if $x_0\in M_\theta\cap \Omega$, $p_m(x)$ is a homogeneous harmonic polynomial of degree bigger than or
equal to $2$.
 Hence, we have
\vspace{.1cm}

$\bullet$ {\em Property 3: around any point $x\in M_\theta\cap \Omega$, $N_\theta$ consists of at least two $C^2$ curves intersecting transversally at $x$.}

\vspace{.1cm}

 Finally, by the stability of $u_\varepsilon$, we claim:

\vspace{.1cm}

$\bullet$ {\em Property 4: $N_\theta$ cannot ``enclose'' any sub-domain of $\Om$. More precisely, if $W\subset \Om$ is a domain, then $\partial W \not \subset N_\theta$. Here $\partial W$ denotes the boundary of $W$ as a subset of ${\R}^2$, and we assume that $W\neq \emptyset$.}

\vspace{.1cm}
Indeed, if $\partial W \subset N_\theta$, then $|\Om \backslash W|>0$ by {Property 2}. By the monotonicity of the first eigenvalue with respect to domains, the first eigenvalue of $L_\varepsilon(u_\varepsilon)$ in $W$ is positive. On the other hand, we have
\begin{equation*}
\begin{cases}
-div\Big[\big(\e^2 +|\nabla u_\varepsilon|^2\big)^\frac{p-2}2\nabla u_{\varepsilon,\theta}+
 (p-2) \big(\e^2 +|\nabla u_\varepsilon|^2\big)^\frac{p-4}2(\nabla u_\varepsilon \nabla u_{\varepsilon,\theta})\nabla u_\varepsilon
  \Big]\\
  \qquad \qquad\qquad\qquad\qquad\qquad\qquad\qquad\qquad\qquad\quad \quad  \ \  \displaystyle=f'(u_\varepsilon)u_{\varepsilon,\theta}  \quad  in\ W,\\
u_{\varepsilon,\theta}=0 \qquad\qquad\qquad\qquad\qquad\qquad\qquad\qquad\qquad \qquad\qquad\qquad\quad \    on\ \partial W.
\end{cases}
\end{equation*}
We claim that $u_{\varepsilon,\theta} \not \equiv0$ in $W$, if not, $u_{\varepsilon,\theta} \equiv0$ in $W$ implies that  
$u_{\varepsilon,\theta}$ vanishes of infinite order in the interior of $W$, 
which is an absurd.
Moreover, it follows from $u_\varepsilon\in C^{3}(\bar{\Om})$ that  
$u_{\varepsilon,\theta}\in W_0^{1,p} (\Om)\backslash \{0\}$.
So the first eigenvalue of $L_\varepsilon(u_\varepsilon)$ in $W$ is non-positive, a contradiction.

Using Properties $1$ to $4$, we  complete the proof.
\end{proof}

For $u_\varepsilon$ given in Proposition \ref{toget}, consider the map $T: \bar \Om \rightarrow {\R}^2$ given by
\begin{equation*}
 T(x)= (u_{\varepsilon,x_2x_2}u_{\varepsilon,x_1}-u_{\varepsilon,x_1x_2}u_{\varepsilon,x_2},
 u_{\varepsilon,x_1x_1}u_{\varepsilon,x_2}-u_{\varepsilon,x_1x_2}u_{\varepsilon,x_1}).
\end{equation*}
Since $u_\e\in C^{3}(\bar \Om)$, $T$ is of class $C^1$.
\begin{lemma}\label{lem1}
We have  $\mathbf{0}\not\in T(\partial \Om)$ and $\deg (T,\Om,\mathbf{0})=1$.
\end{lemma}
\begin{proof}
Let $x_0=(x_{01},x_{02})\in \Om$ and consider the homotopy
\begin{align*}
  H:[0,1]\times  \bar \Om &\rightarrow {\R}^2\\
  (t,x)&\mapsto tT(x)+(1-t)(x-x_0),
\end{align*}
then $H$ is an admissible homotopy, i.e., $H(t,x)\neq \mathbf{0}$ for any  $t\in [0,1]$ and $x\in \partial \Om$. Otherwise, by direct computations, there exist $\tau\in [0,1]$ and $\bar{x}=(\bar{x}_{1},\bar{x}_{2})\in \partial \Om$ such that
\begin{equation*}
  -\tau\mathfrak{R}(\bar{x})|\nabla u_\varepsilon(\bar{x})|^3=(\tau-1)[(\bar{x}_{1}-x_{01})u_{\varepsilon, x_1}(\bar{x})+
  (\bar{x}_{2}-x_{02})u_{\varepsilon, x_2}(\bar{x})].
\end{equation*}
Write $\nu=(\nu_{x_1},\nu_{x_2})$ for the unit exterior normal vector to $\partial \Om$ at $\bar{x}$, it follows
\begin{equation*}
  -\tau\mathfrak{R}(\bar{x})|\nabla u_\varepsilon(\bar{x})|^3=(\tau-1)\frac{\partial u_\varepsilon}{\partial \nu}(\bar{x})[(\bar{x}_{1}-x_{01})\nu_{x_1}+
  (\bar{x}_{2}-x_{02})\nu_{x_2}].
\end{equation*}
Since $\Om$ is strictly star-shaped with respect to the point $x_0$, we have $(\bar{x}_{1}-x_{01})\nu_{x_1}+
  (\bar{x}_{2}-x_{02})\nu_{x_2}>0$. However, this is impossible, because $\tau\in [0,1]$, $\mathfrak{R}(\bar{x})>0$, $|\nabla u_\varepsilon(\bar{x})|>0$, and   $\frac{\partial u_\varepsilon}{\partial \nu}(\bar{x})<0$.
So we conclude the result.
\end{proof}
Arguing as in \cite[Lemma 2]{DGM}, we can prove the following lemma.
\begin{lemma}\label{lem2}
If $x\in \Om$ is such that $T(x)=\mathbf{0}$, then either
\begin{equation*}
  \text{$x$ is a critical point of $u_\varepsilon$},
\end{equation*}
or
\begin{equation*}
  \text{there exists $\theta\in [0,2\pi)$ such that $x\in M_\theta$}.
\end{equation*}
\end{lemma}
We point out that,  if $x\in M_\theta$, then up to a rotation, we can assume that
\begin{equation}\label{M}
 u_{\varepsilon,x_1}(x)= u_{\varepsilon,x_1x_1}(x) = u_{\varepsilon,x_1x_2}(x) = 0,\quad u_{\varepsilon,x_2}(x)\neq 0.
\end{equation}
If $x$ is an isolated zero of $T$,  for any $r>0$ small enough, we denote by ind$(T,x)=\deg (T,B(x,r),\mathbf{0})$ the Browner degree of $T$ in a ball of ${\R}^2$ centered at $x$ with radius $r$.
\begin{lemma}\label{lem3}
Let $x\in \Om$ be such that $T(x)=\mathbf{0}$, we have

$(i)$ If $x$ is a non-degenerate critical point of $u_\varepsilon$, then  ind$(T,x)=1$.

$(ii)$ If $x\in M_\theta$ for some $\theta\in [0,2\pi)$ and it is a non-degenerate critical point of $u_{\varepsilon,\theta}$, then  ind$(T,x)=-1$.
\end{lemma}
\begin{proof}
The proof of $(i)$ is similar to \cite[Lemma 3]{DGM}.
If $x\in M_\theta$, by \ref{M}, we have
\begin{align*}
  &\det \text{Jac}_T(x)\\
  =&-u_{\varepsilon,x_2}^2(x)
  \big[u_{\varepsilon,x_1x_1x_2}^2(x)-u_{\varepsilon,x_1x_1x_1}(x)u_{\varepsilon,x_1x_2x_2}(x)\big]\\
  =&-u_{\varepsilon,x_2}^2(x)\Big[u_{\varepsilon,x_1x_1x_2}^2(x)+\frac{\varepsilon^2+|\nabla u_\varepsilon|^2+(p-2)u_{\varepsilon,x_2}^2(x)}{\varepsilon^2+|\nabla u_\varepsilon|^2}u_{\varepsilon,x_1x_2x_2}^2(x)\Big],
\end{align*}
where the last equality follows from \ref{f1}. Since $x$ is a non-degenerate critical point of $u_{\varepsilon,\theta}$, then
\begin{equation*}
  u_{\varepsilon,x_1x_1x_2}^2(x)-u_{\varepsilon,x_1x_1x_1}(x)u_{\varepsilon,x_1x_2x_2}(x)\neq0.
\end{equation*}
Combining with $\varepsilon^2+|\nabla u_\varepsilon|^2+(p-2)u_{\varepsilon,x_2}^2(x)>0$, we obtain the result.
\end{proof}

\noindent{\bf \emph{Proof of Theorem \ref{th2}}}\ \ \  By Proposition \ref{prop1}, we have $M_\theta\cap \Omega=\emptyset$ for any $\theta\in [0,2\pi)$.
If  $T(x)=\mathbf{0}$, by Lemmas \ref{lem1} and \ref{lem2}, we know that $x\in \Om$, and it is a critical point of $u_\varepsilon$. Moreover, using Proposition \ref{prop1} again, we have that $x$ is non-degenerate. 

Finally, by Lemmas \ref{lem1} and \ref{lem3}, we have
\begin{equation*}
\sharp \text{\{critical points of $u_\varepsilon$\}}=\sum\limits_{\text{$x\in \Om$ such that $\nabla u_\varepsilon(x)=\mathbf{0}$}}\text{ind}(T,x)=\deg(T,\Om,\mathbf{0})=1,
\end{equation*}
which gives the desired. \qed

\vspace{.3cm}

Obviously, we have the following corollary.
\begin{cor}\label{coro}
Let $D\subset \bar{\Om}$ be such that $\mathbf{0}\not\in T(\partial D)$ and $\deg (T,D,\mathbf{0})=1$. If $M_\theta\cap D=\emptyset$ for any $\theta\in [0,2\pi)$, then $u_\varepsilon$
has exactly one critical point in $D$, which is a non-degenerate maximum point.
\end{cor}

To prove Theorem \ref{th1}, we give the following definition.
\begin{definition}
A set $S\subset \Om$ is said to be a minima (maxima) set of $u$, if there exists $T\supset S$ such that $u|_{T\backslash S}>(<)u|_S$.
\end{definition}

\noindent{\bf \emph{Proof of Theorem \ref{th1}}}\ \ \
Even if $u_\varepsilon$ has only one critical point, pass to the limit as $\varepsilon\rightarrow 0$, the solution $u$ can have many critical points. 
 However, some cases can be excluded.
 Denote the unique critical point of $u_\varepsilon$ by $x_{\varepsilon,max}$, which is a non-degenerate maximum point. Then  
$x_{\varepsilon,max}$ converges to a point $x_{max}$ as $\varepsilon\rightarrow0$, and $x_{max}$ is a  maximum point of $u$.
In view of $\nabla u\neq\mathbf{0}$ on $\partial \Om$,  
 the critical points of $u$ can only be in $\Om$.

\vspace{.1cm}

 $\bullet$ {\em Internal absolute minima}

 Since $u\in C^1(\bar{\Om})$, $u>0$ in $\Om$ and $u=0$ on $\partial \Om$, it is obvious that $u$ has no internal absolute minima in $\Om$.

\vspace{.1cm}

 $\bullet$ {\em Internal relative minima}

\vspace{.1cm}


 {\em Case 1: $x_0$ is an isolated minimum point of $u$, that is, there exists $\delta>0$ such that for any ${x}\in U_{{\delta}}({x}_0)$, $u({x})> u({x}_0)$.} Since  $u_\e\to u$ in $C^1(\bar{\Om})$ as $\varepsilon\rightarrow 0$, for any ${x}\in U_{{\delta}}({x}_0)$,  we have
 \begin{align*}
   u_\varepsilon(x)- u_\varepsilon({x}_0)=&\underbrace{u_\varepsilon(x)- u({x})}_{\rightarrow0}+\underbrace{u(x)-u(x_0)}_{>0}+\underbrace{u(x_0)- u_\varepsilon({x}_0)}_{\rightarrow0}> 0,
 \end{align*}
 as $\varepsilon\rightarrow0$, which implies that ${x}_0$ is a relative minimum point of $u_\varepsilon$, this is a contradiction, because $u_\e$ has only one absolute maximum point.

{\em Case 2: there exists a domain $S\subset \Om$ containing $x_0$ (at least one line)  such that $S$ is a minima set of $u$.} Then there exists a domain $T\supset S$ such that $u|_{T\backslash S}>u|_S=u(x_0)$. Hence, we have
\begin{equation*}
  u_\varepsilon|_{T\backslash S}>u_\varepsilon(x_0),
\end{equation*}
and
\begin{equation*}
  u_\varepsilon|_{S}- u_\varepsilon(x_0)\rightarrow 0,
\end{equation*}
as $\varepsilon\rightarrow0$.
This proves that $u_\varepsilon$ admits a relative minimum point in $S$, a contradiction. 

\vspace{.1cm}

 $\bullet$ {\em Internal relative maxima}

\vspace{.1cm}

 If there exists a relative maximum point ${x_0}$ of $u$ with $u({x_0})<||u||_\infty$, similar as above, there exists $x_\varepsilon \in \Om$ 
 which locally
 maximizes $u_\e$.
Moreover, either $x_\varepsilon=x_0$ or $x_\varepsilon \in S$ (where $S$ is a maxima set of $u$)  such that $u_\varepsilon(x_\varepsilon)- u_\varepsilon(x_0)\rightarrow0$ as $\varepsilon\rightarrow 0$.
So
we have
 \begin{align*}
   u_\e({x_\varepsilon})-||u_\e||_\infty=&\underbrace{u_\varepsilon(x_\varepsilon)-u_\varepsilon(x_0)}_{=0 \ \text{or} \ \rightarrow0}+\underbrace{u_\varepsilon(x_0)-u(x_0)}_{\rightarrow0}\\
   &+\underbrace{u(x_0)-||u||_\infty}_{<0}+\underbrace{||u||_\infty-||u_\varepsilon||_\infty }_{\rightarrow0}<0, \quad \text{as $\varepsilon\rightarrow0$},
 \end{align*}
 which implies $u_\e({x_\varepsilon})<||u_\e||_\infty$,
 a contradiction.

\vspace{.1cm}

 $\bullet$  {\em Saddle point}

\vspace{.1cm}

 If there exists an {isolated} saddle point ${x_0}$ of $u$, then $x_0$ is unstable, i.e., ind$(\nabla u,x_0)=0$. Otherwise, by the definition of the index, $u_\varepsilon$ has a critical point  $x_\varepsilon$ close to $x_0$,
 leading to a contradiction.

\vspace{.1cm}

 {\bf Therefore, the critical points of $u_\varepsilon$ has only the internal absolute maxima and possibly saddle points with zero index.} Denote $Argmax(u)$ the set of all absolute maxima for $u$. By the result of \cite[Theorem 1.1]{L} or \cite[Corollary 1.7]{ACF}, the Lebesgue measure of $Argmax(u)$ is zero.

Note that, from the non-degeneracy of the maximum point $x_{\varepsilon,max}$, $u_\varepsilon$ is strictly concave
in the domain close to $x_{\varepsilon,max}$. In particular, the super-level sets of $u_\varepsilon$ are strictly convex
in this domain. More rigorously, for any $\iota> 0$ small enough, there exists a small parameter $\delta_\iota>0$ such that
\begin{equation}\label{5}
  u_\e(\lambda x+(1-\lambda)y)> \lambda u_\e(x)+(1-\lambda)u_\e(y),
\end{equation}
for any  $x,y\in U_\iota$ and $\lambda\in (0,1)$, where
\begin{equation}\label{6}
U_\iota=\{x\in \Om: u_\e(x)>||u_\e||_\infty-\delta_\iota\}
\end{equation}
is strictly convex.

\vspace{.1cm}

$\bullet$ {\em Case 1: $\delta _\iota\to0$ as $\iota\rightarrow 0$.}

\vspace{.1cm}

Passing to the limit and using the uniform convergence of $u_\varepsilon$ to $u$, we deduce from \ref{5} and \ref{6} that
\begin{equation*}
  u(\lambda x+(1-\lambda)y)\ge \lambda u(x)+(1-\lambda)u(y),\quad\hbox{for any $x,y\in \tilde{U}$ and $\lambda\in (0,1)$},
\end{equation*}
where
\begin{equation*}
\tilde{U}=\{x\in \Om: u(x)=||u||_\infty\}.
\end{equation*}
For any $x,y\in Argmax(u)=\tilde{U}$ and $\lambda\in (0,1)$, we have
\begin{equation*}
 ||u||_\infty\geq  u(\lambda x+(1-\lambda)y)\ge \lambda u(x)+(1-\lambda)u(y)=||u||_\infty,
\end{equation*}
which means $\lambda x+(1-\lambda)y \in Argmax(u)$.
So $Argmax(u)$ is convex. 

We claim that $Argmax(u)$ is a point or segment.
If $Argmax(u)$ contains only one point, then the claim follows. If $Argmax(u)$ contains at least two points $x$ and $y$, denote the segment connecting $x$ and $y$ by $[x,y]$, then $Argmax(u)=[x,y]$. In fact, 
 for any $z\in [x,y]$, we have
\begin{equation*}
  ||u||_\infty\geq u(z) \geq \min\{u(x),u(y)\}=||u||_\infty,
\end{equation*}
this gives $[x,y]\subset Argmax(u)$. On the other hand, if there exists another point $\xi\in Argmax(u)$ but $\xi\not \in [x,y]$, then by the convexity of $Argmax(u)$ and $[x,y]\subset Argmax(u)$, we obtain
\begin{equation*}
 \Delta xy\xi \subset Argmax(u),
\end{equation*}
which is a contradiction,
since the Lebesgue measure of $Argmax(u)$ is zero. Thus, $Argmax(u)= [x,y]$.

\vspace{.1cm}

$\bullet$ {\em Case 2: $0<\delta<\delta_\iota<2\delta$ for some fixed small $\delta>0$.}

\vspace{.1cm}

It follows from \ref{5} and \ref{6} that
\begin{equation*}
  u(\lambda x+(1-\lambda)y)\ge \lambda u(x)+(1-\lambda)u(y),\quad\hbox{for any $x,y\in \bar{U}$ and $\lambda\in (0,1)$},
\end{equation*}
where
\begin{equation*}
\bar{U}=\{x\in \Om: u(x)\geq ||u||_\infty-\delta_*\},
\end{equation*}
and $\delta_*=\lim\limits_{\iota\rightarrow0} \delta_\iota \in [\delta,2\delta]$.  For any $x,y\in Argmax(u)\subset \bar{U}$ and $\lambda\in (0,1)$, we have
\begin{equation*}
 ||u||_\infty\geq  u(\lambda x+(1-\lambda)y)\ge \lambda u(x)+(1-\lambda)u(y)=||u||_\infty,
\end{equation*}
which means $\lambda x+(1-\lambda)y \in Argmax(u)$.
So $Argmax(u)$ is convex, and $Argmax(u)$ is a point or segment.
\qed

\section{Proof of Theorem \ref{th2'}}\label{sec4}

First, we consider the case where the curvature of the boundary vanishes at isolated points.
By the compactness of $\partial \Om$ and $u_\varepsilon\in C^2(\bar{\Om})$, we know that the curvature vanishes
only at finitely many points of $\partial \Om$. Without loss of generality, we assume that $\Om$ is a smooth bounded domain such that the curvature is zero at a single point of the boundary and positive elsewhere. Up to a rotation and translation, we assume $\Om\subset \{(x_1,x_2)\in {\R}^2:x_2<0\}$ such that $\partial \Om$ is tangent to the $x_1$-axis at $0$, 
which is the only point where the curvature is zero. 

For any fixed $\varepsilon\in (-\varepsilon_0,\varepsilon_0)$, let $u_\varepsilon$ be given in Proposition \ref{toget}, then
\begin{thm}\label{th3}
Under the assumptions of Theorem \ref{th2'}, $u_\varepsilon$ has a unique critical point, which is a non-degenerate maximum point. 
\end{thm}

Similar to Proposition \ref{prop1}, we have
\begin{prop}\label{prop2}
Under the assumptions of Theorem \ref{th2'}, 
for
any $\theta\in [0,2\pi)$, the nodal set $N_\theta$ is a $C^2$ curve in $\bar \Om$ without self-intersection, which hits $\partial \Om$ at the two end points of $N_\theta$.
Moreover, in any critical point of $u_\varepsilon$, the Hessian has rank $2$.
\end{prop}
\begin{proof}
For any $\theta\in (0,\pi)\cup (\pi,2\pi)$, with a similar argument of Proposition \ref{prop1}, we obtain the result.

 For $\theta=0$ or $\pi$,  we  can prove that  $N_\theta\cap \partial \Om$ consists of exactly two points, and the {Properties 1,3,4} in the proof of Proposition \ref{prop1} still remain true. Finally, using a similar proof of \cite[Proposition 1]{DGM}, we obtain {Property 2}, which completes the proof. 
\end{proof}
To prove Theorem \ref{th3}, we need the following auxiliary lemma.
\begin{lemma}\label{auxi}
Under the assumptions of Theorem \ref{th2'}, suppose that $u_{\varepsilon,x_1x_2}(\mathbf{0})=0$, then $\mathfrak{R}_{x_2}(\mathbf{0})<0$.
\end{lemma}
\begin{proof}
From $\mathfrak{R}(\mathbf{0})=0$ and $u_{\varepsilon,x_1}(\mathbf{0})=0$, we obtain $u_{\varepsilon,x_1x_1}(\mathbf{0})=0$.
By a direct computation, we obtain
\begin{equation*}
  \mathfrak{R}_{x_2}(\mathbf{0})=-\frac{u_{\varepsilon,x_1x_1x_2}(\mathbf{0})u_{\varepsilon,x_2}^2(\mathbf{0})}{|\nabla u_\varepsilon|^3}.
\end{equation*}
We claim that $u_{\varepsilon,x_1x_1x_2}(\mathbf{0})>0$, then we complete the proof, since $u_{\varepsilon,x_2}(\mathbf{0})=\frac{\partial u_\varepsilon}{\partial \nu}(\mathbf{0})<0$.
To prove this claim, we divide the proof into four steps.

\vspace{.1cm}

$\bullet$ {\em Step 1: $u_{\varepsilon,x_1x_1x_2}(\mathbf{0})\neq0$.}

\vspace{.1cm}

Since $u_{\varepsilon,x_2}(\mathbf{0})\neq0$, by the implicit function theorem, we get that, around the point $\mathbf{0}$, $u_\varepsilon(x_1,x_2)=0$ if and only if $x_2=\varphi(x_1)$ for some function $\varphi\in C^3(\bar{\Om})$. By the assumptions
on the boundary of $\Om$, we have
\begin{equation*}
  \varphi ({0})=\varphi'({0})=\varphi''({0})=0.
\end{equation*}
In addition, by $\mathfrak{R}'(x_1,\varphi(x_1))|_{x_1=0}=0$ and
\begin{equation*}
\mathfrak{R}(x_1,\varphi(x_1))=\frac{\varphi''(x_1)}{\big[1+(\varphi'(x_1))^2\big]^{\frac{3}{2}}},
\end{equation*}
we obtain
${\varphi'''({0})}=0$.
Differentiating $u_\varepsilon(x_1,\varphi(x_1)) = 0$, we deduce that $u_{\varepsilon,x_1x_1x_1}(\mathbf{0})=0$. Moreover, it follows from \ref{f1} that 
\begin{align*}
  &-\big(\e^2 +|\nabla u_\varepsilon|^2\big)^\frac{p-2}2(u_{\varepsilon,x_1x_1x_1}+u_{\varepsilon,x_1x_2x_2}) \nonumber\\
  &-(p-2) \big(\e^2 +|\nabla u_\varepsilon|^2\big)^\frac{p-4}2
  (u_{\varepsilon,x_1}u_{\varepsilon,x_1x_1}+u_{\varepsilon,x_2}u_{\varepsilon,x_1x_2})
  \Delta u_\varepsilon \nonumber\\
  &-(p-2)(p-4) \big(\e^2 +|\nabla u_\varepsilon|^2\big)^\frac{p-6}2
  (u_{\varepsilon,x_1}u_{\varepsilon,x_1x_1}+u_{\varepsilon,x_2}u_{\varepsilon,x_1x_2}) \nonumber\\
  &\ \ \ \times  \sum_{i,j=1}^2
  u_{\varepsilon,x_i}u_{\varepsilon,x_j}u_{\varepsilon,x_ix_j}
  \\
  &-2(p-2) \big(\e^2 +|\nabla u_\varepsilon|^2\big)^\frac{p-4}2\sum_{i,j=1}^2
  u_{\varepsilon,x_1x_i}u_{\varepsilon,x_j}u_{\varepsilon,x_ix_j}
\nonumber \\
  &-(p-2) \big(\e^2 +|\nabla u_\varepsilon|^2\big)^\frac{p-4}2\sum_{i,j=1}^2
  u_{\varepsilon,x_i}u_{\varepsilon,x_j}u_{\varepsilon,x_1x_ix_j}=f'(u_\varepsilon)u_{\varepsilon,x_1}\nonumber,
\end{align*}
which implies $u_{\varepsilon,x_1x_2x_2}(\mathbf{0})=0$. If $u_{\varepsilon,x_1x_1x_2}(\mathbf{0})=0$, then the Taylor expansion of $u_{\varepsilon,x_1}$ in a neighborhood of $\mathbf{0}$ becomes
\begin{equation*}
  u_{\varepsilon,x_1}(x)=\text{homogeneous harmonic polynomial of order three} +O(|x|^4).
\end{equation*}
So around the point $\mathbf{0}$,  the nodal curve  $\tilde{N}=\{x\in \bar{\Om}:u_{\varepsilon,x_1}(x)=0\}$  consists of at least three curves 
intersecting $\partial \Om$ at $\mathbf{0}$, and at least two must be contained in $\Om$, a contradiction with Proposition \ref{prop2}.


\vspace{.1cm}

$\bullet$ {\em Step 2: parametrization of $\tilde{N}$ near $\mathbf{0}$.}

\vspace{.1cm}

{ Let $F(x_1,x_2)=u_{\varepsilon,x_1}(x_1,x_2)$ with $(x_1,x_2)\in \bar{\Om}$, then
\begin{equation*}
  F(\mathbf{0})=F_{x_1}(\mathbf{0})=F_{x_2}(\mathbf{0})=F_{x_1x_1}(\mathbf{0})=F_{x_2x_2}(\mathbf{0})=0,\quad { F_{x_1x_2}(\mathbf{0})\neq0}.
\end{equation*}
Using the Taylor expansion, in a neighborhood of $\mathbf{0}$,
we have
\begin{equation*}
  F(x_1,x_2)=F_{x_1x_2}(\mathbf{0})x_1x_2+o(x_1^2+x_2^2).
\end{equation*}
Therefore, $F(x_1,x_2)=0$ if and only if  $(x_1,x_2)$ closes to the $x_1$-axis or $x_2$-axis.
By Proposition \ref{prop2}, $\tilde{N}$ consists of
exactly one branch entering in $\Om$ from $\mathbf{0}$. Thus, in a neighborhood of $\mathbf{0}$, $\tilde{N}$ closes to the $x_2$-axis.
That is, there exists small $\delta>0$ such that
$\tilde{N}$
can be parameterized as
\begin{equation*}
 \mathcal{C}=
 \begin{cases}
x_1=g(t),\\
x_2=t,
\end{cases} t\in [-\delta,0].
\end{equation*}
 {Moreover, $g'(0)=0$}}.
\vspace{.1cm}

$\bullet$ {\em Step 3: $u_{\varepsilon,x_1x_1}(g(t),t)\leq 0$ for $t\in [-\delta,0]$.}

\vspace{.1cm}

Let $\bar{x}=(\bar{x}_1,\bar{x}_2)\in \partial \Om$ close to $\mathbf{0}$ with $\bar{x}_1 <0$,  and $(g(\bar{x}_2),\bar{x}_2)\in \mathcal{C}$. Then 
 for any $\bar{x}_1\leq x_1<g(\bar{x}_2)$, we have $u_{\varepsilon,x_1}({x}_1,\bar{x}_2)>0$ 
 and $u_{\varepsilon,x_1}(g(\bar{x}_2),\bar{x}_2)=0$, thus $u_{\varepsilon,x_1x_1}(g(\bar{x}_2),\bar{x}_2)\leq 0$.

\vspace{.1cm}

$\bullet$ {\em Step 4: end of the proof.}

\vspace{.1cm}

Set $H(t)=u_{\varepsilon,x_1x_1}(g(t),t)$ for $t\in [-\delta,0]$, by {Step 3} and the assumptions of $\Om$, we have $H(0)=0$ and $H(t)\leq 0$. Hence, $H'(0)\geq 0$. Since $H'(t)=u_{\varepsilon,x_1x_1x_1}(g(t),t)g'(t)+u_{\varepsilon,x_1x_1x_2}(g(t),t)$,  { $u_{\varepsilon,x_1x_1x_1}(\mathbf{0})=0$, and $g'(0)=0$}, we have $u_{\varepsilon,x_1x_1x_2}(\mathbf{0}) \geq0$. This with {Step 1}  gives the claim.
\end{proof}

\noindent{\bf \emph{Proof of Theorem \ref{th3}}}\ \ \
As mentioned above, we have $u_{\varepsilon,x_1}(\mathbf{0})=0$, $u_{\varepsilon,x_2}(\mathbf{0})<0$ and $u_{\varepsilon,x_1x_1}(\mathbf{0})=0$. We distinguish the two cases, according to whether $u_{\varepsilon,x_1x_2}(\mathbf{0})$ vanishes or not.

\vspace{.1cm}

$\bullet$ {\em Case 1: $u_{\varepsilon,x_1x_2}(\mathbf{0})\neq0$.}

\vspace{.1cm}

Similar to Lemma \ref{lem1}, we have $\mathbf{0}\not\in T(\partial \Om)$ and $\deg (T,\Om,\mathbf{0})=1$. By Proposition \ref{prop2}, we know that $M_\theta \cap \Om=\emptyset$ for any $\theta\in [0,2\pi)$.
Applying Corollary \ref{coro}, we get the result.

\vspace{.1cm}   

$\bullet$ {\em Case 2: $u_{\varepsilon,x_1x_2}(\mathbf{0})=0$.}

\vspace{.1cm} 

In this case, a direct computation yields $T(\mathbf{0})=\mathbf{0}$ but $\mathbf{0}\in \partial \Om$.  Thus the degree of $T$ is not well defined.
For any $\rho>0$ small enough, we define $\Om_\rho=\Om\backslash \bar{U}_\rho$, where $U_\rho$ is a ball in $\mathbb{R}^2$
centered at $\mathbf{0}$ with radius $\rho$, it is chosen in such a way that
$\nabla u_\varepsilon\neq \mathbf{0}$ in $\bar{\Om}\cap \bar{U}_\rho$, and
$\Om_\rho$ is strictly star-shaped with respect to some point $x_0=(x_{01},x_{02})\in \Om_\rho$.
Now, we consider the map $T: \bar \Om_\rho \rightarrow {\R}^2$, then the degree of $T$ is well defined, and if the homotopy
\begin{align*}
  H_\rho:[0,1]\times  \bar \Om_\rho &\rightarrow {\R}^2\\
  (t,x)&\mapsto tT(x)+(1-t)(x-x_0),
\end{align*}
is admissible, we have $\deg (T,\Om_\rho,\mathbf{0})=1$. Assume by contradiction, similar to the proof of Lemma \ref{lem1}, there exist $t_\rho\in [0,1]$ and $x_\rho=(x_{\rho1},x_{\rho2})\in \partial \Om_\rho$ such that
\begin{equation*}
  -t_\rho\mathfrak{R}(x_\rho)|\nabla u_\varepsilon(x_\rho)|^3=(t_\rho-1)[(x_{\rho1}-x_{01})u_{\varepsilon, x_1}(x_\rho)+
  (x_{\rho2}-x_{02})u_{\varepsilon, x_2}(x_\rho)].
\end{equation*}
Writing $\nu=(\nu_{x_1},\nu_{x_2})$ for the unit exterior normal vector to $\partial \Om_\rho$ at $x_\rho$, by continuity, we have $|\nabla u_\varepsilon(x_\rho)|>0$ and
\begin{align*}
  &(x_{\rho1}-x_{01})u_{\varepsilon, x_1}(x_\rho)+
  (x_{\rho2}-x_{02})u_{\varepsilon, x_2}(x_\rho)\\
  =&\frac{\partial u_\varepsilon}{\partial \nu}({x_\rho})[({x}_{\rho1}-x_{01})\nu_{x_1}+
  ({x}_{\rho2}-x_{02})\nu_{x_2}]<0,\quad \text{as $\rho\rightarrow0$}.
\end{align*}
  Thus $\mathfrak{R}(x_\rho)\leq 0$ and $x_\rho \in \Om \cap \partial U_\rho$. Then the vertical line $x_1=x_{\rho 1}$ hits $\partial \Om$ at a unique point $y=(y_1,y_2)$ with $y_1=x_{\rho1}$ and $y_2>x_{\rho2}$. Since $\mathfrak{R}(y)\geq 0$ and $\mathfrak{R}(x_\rho)\leq 0$, there exists $z_\rho\in \bar{\Om} \cap \bar{U}_\rho$ such that $\mathfrak{R}_{x_2}(z_\rho)\geq 0$, and as $\rho\rightarrow0$, we have $\mathfrak{R}_{x_2}(\mathbf{0})\geq 0$, which contradicts Lemma \ref{auxi}. Moreover, it follows from Proposition \ref{prop2} that $M_\theta\cap \Om_\rho=\emptyset$ for any $\theta\in [0,2\pi)$.
  Thus, apply Corollary \ref{coro}, the claim follows.
\qed

\vspace{.5cm}

We now treat domains where the curvature vanishes at some segments of its boundary. Similarly, we know that the curvature vanishes
only at finitely many segments. Without loss of generality,
we assume $\Om\subset \{(x_1,x_2)\in {\R}^2:x_2<0\}$ is a smooth bounded domain such that $\partial \Om$ is tangent to the $x_1$-axis,
and
the curvature is zero at an interval
$\Gamma=\{(x_1,x_2)\in {\R}^2:x_1\in (-1,1),x_2=0\}$.
Then
\begin{thm}\label{th4}
Under the assumptions of Theorem \ref{th2'}, $u_\varepsilon$ has a unique critical point, which is a non-degenerate maximum point.
\end{thm}
\begin{prop}\label{prop3}
Under the assumptions of Theorem \ref{th2'}, 
for
any $\theta\in [0,2\pi)$, there exists a unique point $\xi\in \Gamma$ such that the nodal set $N_\theta \subset \bar{\Om}\backslash\Gamma_\xi$ (where $\Gamma_\xi=\Gamma\backslash \{\xi\}$) is 
a $C^2$ curve without self-intersection, which hits $\partial \Om$ at two points.
Moreover, in any critical point of $u_\varepsilon$, the Hessian has rank $2$.
\end{prop}
\begin{proof}
For any $\theta\in (0,\pi)\cup (\pi,2\pi)$, a similar discussion of Proposition \ref{prop1} gives the result.

  For $\theta=0$ or $\pi$,  
it is obvious that there exists a unique point $p\in N_\theta\cap(\partial \Om\backslash \Gamma)$, and the {Properties 1,3,4} in the proof of Proposition \ref{prop1} still remain true.

We claim that, there exists a unique point $\xi\in \Gamma$ such that, around the point $\xi$, the nodal curve ${N}_\theta$ consists of exactly two curves: the first one is $\Gamma$, while the second 
 intersects $\partial \Om$ transversally at $\xi$. Otherwise, there exists $\eta>0$ such that ${N}_\theta\cap  \Om \cap (\mathbb{R}\times (-\eta,0])= \emptyset$,
then the nodal curve ${N}_\theta$ starting from $p$ has to enclose a non-empty domain $W\subset \Om$, which is a contradiction. Moreover, it cannot have more than one
this type of curves
exiting at $\xi$, otherwise, there exists a sub-domain $W$ as before, and this is a contradiction.  For the uniqueness of $\xi\in \Gamma$, if there exists another point $\varsigma\in \Gamma$ with the same property, we can argue as before to get the existence of $W$, which is impossible.

Up to a translation, we assume $\xi=\mathbf{0}$. 
We point out that, for any $q\in \Gamma$, there holds
\begin{equation*}
  u_{\varepsilon,x_1}(q)=0,\quad u_{\varepsilon,x_2}(q)<0,\quad \text{and \ $u_{\varepsilon,x_1x_1}(q)=0$.}
\end{equation*}
{\bf Moreover, if $q\neq \mathbf{0}$, we have $u_{\varepsilon,x_1x_2}(q)\neq0$,  thus $M_\theta\in \bar{\Om}\backslash \Gamma_0$}. 
Otherwise,
the Taylor expansion of $u_{\varepsilon,x_1}$ in a neighborhood of $q$ becomes
\begin{equation*}
  u_{\varepsilon, x_1}(x)=\text{homogeneous harmonic polynomial of order two} +O(|x|^3).
\end{equation*}
So around the point $q$,  the nodal curve  ${N}_\theta$ consists of at least two curves
intersecting $\partial \Om$ at $q$, and at least one must be entering in $\Om$, a contradiction with the uniqueness of the point $\xi=\mathbf{0}$ with this property.

The 
rest of the proof follows arguing as in \cite[Proposition 1]{DGM}.
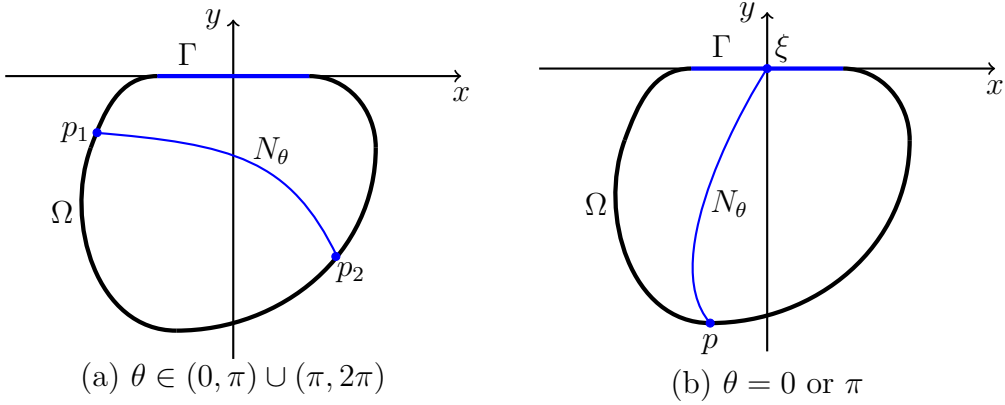
\begin{figure}[htbp]
\begin{minipage}{0.45\textwidth}
\begin{tikzpicture}
	\draw[thick,->] (-3,0) -- (3,0) node[below] {$x$};
	\draw[thick,->] (0,-3.75) -- (0,0.75) node[left] {$y$};
	
	\coordinate (A) at  (-1,0);
	\coordinate (B) at  (-1.875,-0.95);
	\coordinate (C) at  (-0.75,-3.375);
	\coordinate (D) at  (1.875,-0.95);
    \coordinate (E) at  (1,0);

	\draw [ultra thick] (A) to [out=180,in=70] (B);
	\draw [ultra thick] (B) to [out=250,in=180] (C);
	\draw [ultra thick] (C) to [out=0,in=270] (D);
	\draw [ultra thick] (D) to [out=90,in=0] (E);
    \draw [ultra thick, blue] (E) to [out=180,in=0] (A);
	
\node at (-2.25,-1.8) {$\Omega$};
\node at (0.5,-0.99) {$N_\theta$};
\node at (0,-4) {(a) $\theta\in (0,\pi)\cup (\pi,2\pi)$};
\node at (-0.6,0.3) {$\Gamma$};

\node at (-2.075,-0.75) {$p_1$};
\filldraw (-1.792,-0.75) [blue] circle (1.5pt);

\node at (1.55,-2.600) {$p_2$};
\filldraw (1.35,-2.392) [blue] circle (1.5pt);

\draw[thick, blue] (-1.830,-0.75) .. controls (0,-0.9) and (0.75,-1.1) .. (1.37,-2.400);


\end{tikzpicture}
\end{minipage}
\hspace{0.05\textwidth}
\begin{minipage}{0.45\textwidth}
\begin{tikzpicture}
	\draw[thick,->] (-3,0) -- (3,0) node[below] {$x$};
	\draw[thick,->] (0,-3.75) -- (0,0.75) node[left] {$y$};
	
	\coordinate (A) at  (-1,0);
	\coordinate (B) at  (-1.875,-0.95);
	\coordinate (C) at  (-0.75,-3.375);
	\coordinate (D) at  (1.875,-0.95);
    \coordinate (E) at  (1,0);

	\draw [ultra thick] (A) to [out=180,in=70] (B);
	\draw [ultra thick] (B) to [out=250,in=180] (C);
	\draw [ultra thick] (C) to [out=0,in=270] (D);
	\draw [ultra thick] (D) to [out=90,in=0] (E);
    \draw [ultra thick, blue] (E) to [out=180,in=0] (A);
	
\node at (-2.25,-1.8) {$\Omega$};
\node at (-0.5,-1.79) {$N_\theta$};
\node at (0,-4.2) {(b) $\theta=0$ or $\pi$};
\node at (-0.6,0.3) {$\Gamma$};

\node at (0.2,0.27) {$\xi$};
\filldraw (0,0) [blue] circle (1.5pt);

\node at (-0.75,-3.65) {$p$};
\filldraw (-0.75,-3.375) [blue] circle (1.5pt);

\draw[thick, blue] (0,0) .. controls (-0.2,-0.3) and (-1.5,-2.5) .. (-0.75,-3.375);


\end{tikzpicture}
\end{minipage}
\caption{A picture of $N_\theta$ in Proposition \ref{prop3}}
\end{figure}
\end{proof}

\begin{lemma}\label{auxi'}
Under the assumptions of Theorem \ref{th2'}, suppose that $u_{\varepsilon,x_1x_2}(\mathbf{0})=0$, then $\mathfrak{R}_{x_2}(\mathbf{0})<0$.
\end{lemma}
\begin{proof}
We prove that $u_{\varepsilon,x_1x_1x_2}(\mathbf{0})>0$, then $\mathfrak{R}_{x_2}(\mathbf{0})=-\frac{u_{\varepsilon,x_1x_1x_2}(\mathbf{0})u_{\varepsilon,x_2}^2(\mathbf{0})}{|\nabla u_\varepsilon|^3}<0$. Indeed, {since $u_{\varepsilon,x_1x_2}$ is continuous and $u_{\varepsilon,x_1x_2}\neq0$  on $\Gamma\backslash\{\mathbf{0}\}$ by Proposition \ref{prop3}, we have that $u_{\varepsilon,x_1x_2}$ is either strictly positive or  negative
on $\Gamma^-=\{(x_1,x_2)\in \Gamma:x_1\in (-1,0)\}$. In a neighborhood of $(-1,0)$, consider the points $(x_1^-,0)\in \Gamma^-$ and $(x_1^-,x_2)\in \Om$, by the Taylor expansion, we obtain
\begin{equation*}
  u_{\varepsilon,x_1}(x_1^-,x_2)=u_{\varepsilon,x_1x_2}(x_1^-,0)x_2+o(x_1^-+x_2).
\end{equation*}
Since $u_{\varepsilon,x_1}(x_1^-,x_2)>0$ and $x_2<0$, we have $u_{\varepsilon,x_1x_2}(x_1^-,0)<0$. Thus $u_{\varepsilon,x_1x_2}$ is negative
on $\Gamma^-$. Similarly, for $\Gamma^+=\{(x_1,x_2)\in \Gamma:x_1\in (0,1)\}$, we can prove that, $u_{\varepsilon,x_1x_2}$ is positive
on $\Gamma^+$.
Then it follows $u_{\varepsilon,x_1x_1x_2}(\mathbf{0})\geq 0$}, but if equality holds, we can argue as the {Step 1} in the proof
of Lemma \ref{auxi} to get a contradiction.
\end{proof}

\noindent{\bf \emph{Proof of Theorem \ref{th4}}}\ \ \
Repeating the same arguments as the proof of Theorem \ref{th3}, we can complete the proof. The difference is that
$\mathfrak{R}(x_\rho)\leq 0$ implies $x_\rho \in \Om \cap \partial U_\rho$ or $x_\rho\in \Gamma_\rho$, where $\Gamma_\rho=\{(x_1,x_2)\in \Gamma:x_1\in (-1,-\rho)\cup (\rho,1)\}$. In the latter case, we have $\mathfrak{R}(x_\rho)= 0$, thus $t_\rho=1$, i.e., $T(x_\rho)=\mathbf{0}$. However, this is impossible, because $u_{\varepsilon,x_1}(x_\rho)=0$, $u_{\varepsilon,x_2}(x_\rho)<0$, $u_{\varepsilon,x_1x_1}(x_\rho)=0$, and $u_{\varepsilon,x_1x_2}(x_\rho)\neq0$.
\qed

\vspace{.3cm}

\noindent{\bf \emph{Proof of Theorem \ref{th2'}}}\ \ \
The proof is similar to the one of Theorem \ref{th1}, so we omit it here. \qed

\bibliographystyle{abbrv}

\end{document}